\documentclass[letterpaper, 10pt, conference,onecolumn]{ieeeconf}  

\IEEEoverridecommandlockouts                              

\title{\LARGE \bf A stochastic density matrix approach  to approximation of  probability distributions and its application to nonlinear systems$^*$}
\author{Igor G. Vladimirov$^{\dagger}$
\thanks{$^*$This work is supported by the Australian Research Council.}
\thanks{$^\dagger$UNSW Canberra, Australia.
{\tt igor.g.vladimirov@gmail.com}.}
}

\usepackage{amsmath}
\usepackage{amssymb}
\usepackage{amsfonts}
\usepackage{graphicx}
\usepackage{color}
\usepackage{mathptmx} 
\usepackage{times}

\usepackage{datetime}
\newtheorem{lem}{Lemma}
\newtheorem{theorem}{Theorem}

\def\fS{\mathfrak{S}}         
\def\Span{\mathrm{span}}         
\def\rank{\mathrm{rank}}         
\def\div{\mathrm{div}}         
\def\<{\leqslant}           
\def\>{\geqslant}           

\def\mZ{\mathbb{Z}}    
\def\mR{\mathbb{R}}    
\def\mC{\mathbb{C}}    
\def\mH{\mathbb{H}}    
\def\cH{\mathcal{H}}    
\def\cQ{\mathcal{Q}}    
\def\Tr{\mathrm{Tr\,}}   
\def\rT{\mathrm{T}}    
\def\vec{\mathrm{vec}}    

\def\bE{\mathbf{E}}    

\def\bra{\langle}
\def\ket{\rangle}

\def\Bra{\left\langle }
\def\Ket{\right\rangle }

\def\re{\mathrm{e}}    

\def\rd{\mathrm{d}}    

\def\d{\partial}    

\def\rB{\mathrm{B}}    

\def\cL{\mathcal{L}}

\def\bR{\mathbf{R}}

\def\diag{\mathop{\mathrm{diag}}}    

\def\x{\times}
\def\ox{\otimes}

\def\b1{\mathbf{1}}

\def\cG{{\mathcal G}}

\def\cB{\mathcal{B}}
\def\cE{\mathcal{E}}

\def\cA{\mathcal{A}}

\def\mT{\mathbb{T}}
\def\mZ{\mathbb{Z}}

\begin{document}
\maketitle

\thispagestyle{empty}
\pagestyle{plain}

\begin{abstract}
This paper outlines an approach to the approximation of probability density functions by quadratic forms of weighted orthonormal basis functions with positive semi-definite Hermitian matrices of unit trace.  Such matrices are called stochastic density matrices in order to reflect an analogy with the quantum mechanical density matrices. The SDM approximation of a PDF satisfies the normalization condition and is nonnegative everywhere  in contrast to the truncated Gram-Charlier and Edgeworth expansions. For bases with an algebraic structure, such as  the Hermite polynomial and Fourier bases, the SDM approximation can be chosen so as to satisfy given moment specifications and can be optimized using a quadratic proximity criterion. We apply the SDM approach to the Fokker-Planck-Kolmogorov PDF dynamics of Markov diffusion processes governed by nonlinear stochastic differential equations. This leads to an ordinary differential equation for the SDM dynamics of the approximating PDF. As an example, we consider the Smoluchowski SDE on a multidimensional torus.
\end{abstract}

\section{INTRODUCTION}\label{sec:intro}

Practical solvability of performance analysis and control design problems for stochastic systems often depends on tractability of relevant quantities,  such as moments of the system variables. For example, the Kalman filtering and Linear Quadratic Gaussian control theories  \cite{AM_1989} substantially employ the preservation of Gaussian nature of probability distributions of the state variables governed by linear SDEs.
Under the linear dynamics,  the first and second order moments of the variables (and more complicated functionals of Gaussian distributions) are amenable to a complete analysis. It is the convenience of linear Gaussian models that makes  them so popular in filtering/control with quadratic  and related (for example, risk-sensitive) performance criteria. These advantages motivate the  approximation of a nonlinear stochastic system by an effective linear model which underlies the stochastic linearization techniques. The latter date back to \cite{B_1953,C_1963,K_1956}
and have recently been extended to quantum stochastic systems \cite{VP_2012a}.

A different approach to computing  the statistical characteristics of a nonlinear stochastic system (oriented at approximating probability distributions rather than system dynamics) consists, for example,  in representing the probability distribution of its state variables as a mixture of Gaussian distributions whose parameters evolve in time. In fact, mixed Gaussian distributions arise as exact posterior probability distributions in the Lainiotis multimodel filter \cite{L_1975}, where the conditional Gaussian distributions from partial Kalman filters are weighted by recursively updated posterior  probabilities of the corresponding linear models conditioned on the observations. This combination of a bank of Kalman filters with a ``mixing'' block is a recursive implementation of  the Bayesian approach. The important property that the resulting mixture of Gaussian PDFs 
is a legitimate PDF, which is nonnegative everywhere  and satisfies the normalization condition, does not always come with  PDF approximations, in general. For example, the truncated Gram-Charlier and Edgeworth expansions \cite{BC_1989}, based on the Hermite polynomials,  are not equipped with this feature, although they provide control over moments or cumulants up to an arbitrary given order.

The aim of the present paper is to outline an approach to the approximation of PDFs by using quadratic forms of weighted    complex-valued orthonormal basis functions with positive semi-definite Hermitian matrices of unit trace.  These matrices are called stochastic density matrices (SDM) in order to emphasize an analogy (and, at the same time, avoid confusion) with the quantum mechanical density matrices \cite{S_1994}. The SDM approximation leads to a legitimate PDF which is nonnegative everywhere and satisfies the normalization condition. Furthermore, it retains the possibility to control the moments of the PDF for orthonormal bases with an algebraic structure, such as the Hermite polynomial and Fourier bases. The SDM approximation can be optimized by using a proximity criterion for PDFs based on the second-order relative Renyi entropy \cite{R_1961}, which leads to a quadratic minimization problem.

This allows the SDM approach to be applied to PDFs of Markov diffusion processes,  governed by nonlinear SDEs, by reducing the approximate numerical integration of the Fokker-Planck-Kolmogorov equation (FPKE) \cite{S_2008} to the solution of an 
ODE for the SDM, which resembles the Galerkin approximations for parabolic PDEs \cite{E_1998}.
As an illustration, we consider a Smoluchowski SDE \cite{GS_2004,KS_1991,Z_2001} on a multidimensional torus, which provides  an example of a nonlinear stochastic system with rotational degrees of freedom.
The SDM approach admits a real  version in the case of real-valued basis functions, with an appropriate reformulation of the results. It is relevant to mention a connection of this approach with the methods using the sum of squares (SOS) of polynomials for Lyapunov stability analysis and global optimization \cite{L_2001,P_2000}. However,  the SDM approach, proposed in the present paper,  serves a different purpose here and 
is not restricted to polynomials.

The paper is organised as follows. Section~\ref{sec:SDM} describes the class of PDFs generated by an SDM and a set of orthonormal functions. Section~\ref{sec:alg} relates the algebraic structure of the orthonormal basis to the moments of such PDFs. Section~\ref{sec:eff_par} discusses effective parameters of the SDM which control the PDF. Section~\ref{sec:Herm} specifies this class of PDFs for the multivariate Hermite polynomial and Fourier bases. Section~\ref{sec:SDMapp} describes the SDM approximation of a given PDF using a quadratic criterion. Section~\ref{sec:SDMdyn} extends the SDM approximation to PDF dynamics of Markov processes. Section~\ref{sec:SSDE} considers the Smoluchowski SDE on a multidimensional torus.  Section~\ref{sec:spat} reformulates the corresponding FPKE in the spatial frequency domain. Section~\ref{sec:smolSDM} describes the SDM approximation of the PDF dynamics using the Fourier basis and provides numerical results. Section~\ref{sec:conc} makes concluding remarks.

\section{STOCHASTIC DENSITY MATRIX}\label{sec:SDM}

Suppose $G$ is a finite-dimensional state space of a dynamical system. To be specific, we assume that $G$ is a domain in $\mR^n$, or an $n$-dimensional torus $\mT^n$. In what follows, we use a complex Hilbert space
\begin{align}
    \cH  := \cL^2(G,\nu)
\label{cH}
       := \Big\{\!f:G \to \mC: \int_G |f(x)|^2\nu(x)\rd x<+\infty\!\Big\}
\end{align}
of square integrable complex-valued functions on the set $G$ with a weight $\nu: G\to (0,+\infty)$. The norm  $\|f\|_{\cH}:= \sqrt{\bra f,f\ket_{\cH}}$ in the space $\cH$ is generated by
the inner product
\begin{equation}
\label{braket}
    \bra f,g\ket_{\cH}:=     \int_G
    \overline{f(x)}g(x)\nu(x)\rd x,
\end{equation}
where the integral is over the standard $n$-variate Lebesgue measure, and $\overline{(\cdot)}$ is the complex conjugate.  Furthermore, we denote by
$
    \bra X,Y\ket:= \Tr(X^*Y)
$
the Frobenius inner product \cite{HJ_2007} of complex matrices $X$ and $Y$, which reduces to  $\Tr(X Y)$ for complex Hermitian matrices (the real Hilbert space of such matrices of order $r$ is denoted by $\mH_r$, with $\mH_r^+$ the corresponding set of positive semi-definite matrices). Also, $(\cdot)^*:= (\overline{(\cdot)})^{\rT}$ is the complex conjugate transpose, and vectors are organised as columns  unless specified otherwise.
Suppose $(\varphi_k)_{k \in \Omega}\in \cH$ is a fixed but otherwise arbitrary orthonormal basis in the Hilbert space (\ref{cH}), whose elements are indexed by a  denumerable set $\Omega$. The latter is assumed to be a subset of a multidimensional integer lattice which contains the origin: $0\in \Omega$. Therefore,
\begin{equation}
\label{ortho}
    \bra \varphi_j, \varphi_k \ket_{\cH}
    =
    \delta_{jk},
    \qquad
    j,k\in \Omega,
\end{equation}
where
$\delta_{jk}$ is the Kronecker delta. Now, suppose $\Lambda$ is a given nonempty finite subset of $\Omega$ consisting of $N:= \# \Lambda$ elements. Also,
let $S:= (s_{jk})_{ j,k\in \Lambda} \in \mH_N^+$ be a 
matrix  of unit trace $\Tr S:= \sum_{k\in \Lambda} s_{kk} = 1$. Such matrices form a convex subset of the space $\mH_N$ which we denote by
\begin{equation}
\label{fS}
  \fS_N:=
  \left\{
    S\in \mH_N^+:\ \Tr S = 1
  \right\}.
\end{equation}
With the orthonormal basis $(\varphi_k)_{k \in \Omega}$ in $\cH$ and the set $\Lambda$ being fixed, we associate with $S \in \fS_N$
a real-valued function $p:G\to \mR$ defined as a quadratic form
\begin{align}
\nonumber
        p(x)
        := &
        \nu(x)
        \sum_{j,k\in \Lambda}
        s_{jk}
        \overline{\varphi_j(x)}
        \varphi_k(x)
        =
        \nu(x)
        \Phi(x)^*
        S
        \Phi(x)\\
\label{p}
        =&
        \nu(x)
        \Bra
            S,
            \Phi(x)\Phi(x)^*
        \Ket,
        \qquad
        x \in G,
\end{align}
where $\Phi: G \to \mC^N$ is a vector-valued map
formed from the basis functions as
\begin{equation}
\label{Phi}
    \Phi(x)
    :=
    (\varphi_k(x))_{k \in \Lambda}.
\end{equation}

\begin{lem}
\label{lem:pS}
For any 
matrix $S\in \fS_N$ from the set (\ref{fS}), the corresponding function $p$, defined by (\ref{p}) and (\ref{Phi}), is a PDF on the set   $G$. 
\end{lem}
\begin{proof}
The positive semi-definiteness  $S\succcurlyeq 0$ and the positiveness $\nu> 0$ in (\ref{p}) imply that $p(x)\> 0$ for all $x\in G$. Furthermore, the orthonormality  (\ref{ortho}) and the assumption $\Tr S = 1$ imply that the function $p$ satisfies the normalization condition
$
    \int_G
    p(x)\rd x
 =
    \sum_{j,k\in \Lambda}
    s_{jk}
    \bra \varphi_j, \varphi_k\ket_{\cH}
    =    \Tr S
    =1
$,
and hence, $p$ is indeed a PDF.
\end{proof}

Due to its properties and the role in the construction (\ref{p}) of a legitimate PDF, the matrix $S$ resembles the density matrices in quantum mechanics  \cite{M_1995,S_1994}. In order to reflect this analogy (and, at the same time, avoid confusion) with the quantum mechanical density matrices and the stochastic matrices of transition probabilities of classical Markov chains \cite{S_1996}, we will refer to $S$ as a \emph{stochastic density matrix} (SDM). Since the set $\fS_N$ of SDMs $S$ in (\ref{fS}) is convex and the PDF $p$ in (\ref{p}) depends linearly on $S$, the set of such PDFs is also convex.
If the Hilbert space in (\ref{cH}) is restricted to real-valued functions, then the SDM becomes a real symmetric matrix, and the results that follow can be appropriately reformulated for this case.

\section{ALGEBRAIC STRUCTURE AND MOMENTS}\label{sec:alg}

For what follows, we assume that $\varphi_0=1$, and hence, in view of (\ref{ortho}), the weighting function $\nu$ satisfies $\int_G \nu(x)\rd x = 1$  and is itself a PDF. This allows the inner product (\ref{braket}) to be represented as
$\bra f, g\ket_{\cH}=\bE_{\nu}(\overline{f}g)
$,
where $\bE_{\nu}(\cdot)$ denotes the expectation over the PDF $\nu$. Also, suppose the basis functions are algebraically closed in the sense that there exist complex constants $e_{jk\ell}$ such that for any $j,k\in \Omega$, the representation
\begin{equation}
\label{tri}
    \varphi_j(x)
    \overline{\varphi_k(x)}
    =
    \sum_{\ell\in \Omega}
    e_{jk\ell} \varphi_{\ell}(x)
\end{equation}
holds for all $x\in G$ and contains  only a finite number of terms with nonzero coefficients $e_{jk\ell}$.
The coefficients $e_{jk\ell}$ in (\ref{tri}), which quantify the algebraic structure of the orthonormal basis $(\varphi_k)_{k \in \Omega}$,
can be computed as
\begin{equation}
\label{e}
    e_{jk\ell}
    =
    \bra
        \varphi_{\ell},\,
        \varphi_j \overline{\varphi_k}
    \ket_{\cH}
    =
    \bE_{\nu}(
    \varphi_j
    \overline{\varphi_k}
    \overline{\varphi_{\ell}}
    )
\end{equation}
and are symmetric with respect to the subscripts $k$, $\ell$. Furthermore, $e_{jk\ell}$ is real and invariant under arbitrary permutations of its subscripts in the real case mentioned at the end of Section~\ref{sec:SDM}.  For any $\ell \in \Omega$, we define the $\ell$th \emph{structure matrix} $    E_{\ell}
    :=
    (e_{jk\ell})_{j,k\in \Lambda} \in \mC^{N\x N}
$ by
\begin{equation}
\label{E}
    E_{\ell}
    :=
    (e_{jk\ell})_{j,k\in \Lambda}
    =
    \bE_{\nu}
    (
        \overline{\varphi_{\ell}} \Phi \Phi^*
    ),
\end{equation}
where the expectation applies entrywise, and the map $\Phi$ is given by (\ref{Phi}). In particular, since $\varphi_0=1$, then (\ref{e}) implies that $e_{jk0} = \delta_{jk}$ for all $j,k\in \Lambda$, and hence, $E_0 = I_N$ is the identity matrix of order $N$.
Since the left-hand side of (\ref{tri}) is the $(j,k)$th entry of the matrix $\Phi(x)\Phi(x)^*$, the algebraic closedness property  is representable in a vector-matrix form:
\begin{equation}
\label{PhiPhiPhi}
    \Phi(x)\Phi(x)^*
    =
    \sum_{\ell\in \mho}
    \varphi_{\ell}(x)E_{\ell}.
\end{equation}
This series contains only a finite number of terms since the matrix $E_{\ell}$ in (\ref{E}) vanishes for all but finitely many indices $\ell \in \Omega$ which form the set
\begin{equation}
\label{mho}
  \mho := \{\ell\in \Omega: E_{\ell}\ne 0\}.
\end{equation}
This set, consisting of $L:= \# \mho$ elements,   depends on the set $\Lambda$ and contains $0$ (since $E_0 = I_N$).
Substitution of (\ref{PhiPhiPhi}) into the definition of the PDF $p$ in (\ref{p}) leads to
\begin{align}
\nonumber
    p(x)
    & =
    \nu(x)
    \sum_{\ell\in \mho}
    \Bra S, E_{\ell}\Ket
    \varphi_{\ell}(x)\\
\label{pES}
     & =
    \nu(x)
    \sum_{\ell\in \mho}
    \Bra E_{\ell}, S\Ket
    \overline{\varphi_{\ell}(x)}
    =
    \nu(x)
    C(S)^* \Psi(x).
\end{align}
Here, $C: \mH_N \to \mC^L$ is a linear vector-valued map defined in terms of the nonzero structure matrices by
\begin{equation}
\label{C}
  C(S):= (\Bra E_{\ell}, S\Ket)_{\ell \in \mho},
\end{equation}
and, similarly to (\ref{Phi}),  the map $\Psi: G\to \mC^L$ is formed from the basis functions as
\begin{equation}
\label{Psi}
    \Psi(x)
    :=
    (\varphi_{\ell}(x))_{\ell \in \mho}.
\end{equation}
The second equality in (\ref{pES}) follows from $p$ and $\nu$ being real-valued.
In view of this equality and (\ref{ortho}), the expectations of the basis functions over the PDF $p$ (which play the role of generalized moments of  this PDF) are computed as
\begin{equation}
\label{mom}
    \bE_p\varphi_m
    :=
    \int_G
    p(x)
    \varphi_m(x)
    \rd x
     =
    \sum_{\ell\in \mho}
    \bra
        E_{\ell}, S
    \ket
    \bra
        \varphi_{\ell},
        \varphi_m
    \ket_{\cH}
     =
    \bra E_m, S\ket
\end{equation}
for any $m\in \Omega$  and vanish for all $m \in \Omega\setminus \mho$ in view of (\ref{mho}). Therefore, (\ref{pES}) describes the series expansion of the PDF $p$ from (\ref{p}) over the orthonormal basis  in $\cH$, and the map $C$ in (\ref{C}) encodes the dependence of all the nonzero moments in (\ref{mom}) on the SDM $S$. The fact that the moments depend linearly on $S$ allows the SDM to be chosen so as to satisfy given moment specifications for the corresponding PDF $p$ in (\ref{p}) . Note that the moments $\bE_p \varphi_m$, with $m\ne 0$, are responsible for the deviation of the PDF $p$ from the weighting function $\nu$. In particular,  (\ref{pES}) implies that the second-order relative Renyi entropy \cite{R_1961} of $p$ with respect to $\nu$ (with the latter playing the role of a reference PDF) can be expressed in terms of the moments in (\ref{mom}) as
\begin{align}
\nonumber
    \bR(p\| \nu)
    :=&
    \ln
    \int_G
    \frac{p(x)^2}{\nu(x)}
    \rd x
    =
    \ln \bE_{\nu}
    \Big(
    \Big(
        \frac{p}{\nu}
    \Big)^2\Big)\\
\nonumber
    =&
    \ln
    \sum_{\ell,m\in \mho}
    \Bra E_{\ell}, S\Ket
    \Bra S, E_m\Ket
    \Bra
        \varphi_{\ell},
        \varphi_m
    \Ket_{\cH}\\
\label{bR}
    =&
    \ln
    \Big(
        1
        +
        \sum_{\ell\in \mho\setminus \{0\}}
        |\Bra E_{\ell}, S\Ket|^2
    \Big).
\end{align}
Here, the orthonormality condition (\ref{ortho}) is used in combination with the property that $\bra E_0, S\ket = \bra I_N, S\ket = \Tr S = 1$.

\section{EFFECTIVE PARAMETERS}\label{sec:eff_par}

Note that the PDF $p$, defined  by (\ref{p}), can also be represented in the form
\begin{equation}
\label{psi2}
    p(x)
    =
    \sum_{j \in \Lambda}
    \sigma_j
    \big|
    \sqrt{\nu(x)}\,
    \theta_j(x)
    \big|^2
\end{equation}
which resembles the SOS structures \cite{L_2001,P_2000}. Here,
$\sigma_j$ denote the eigenvalues of the SDM $S$ (they are all real and nonnegative and satisfy $\sum_{j\in \Lambda}\sigma_j = 1$), and the functions $\theta_j: G\to \mC$  are obtained by a unitary transformation of the functions $\varphi_k$ as
\begin{equation}
\label{theta}
    \theta_j(x)
    :=
    \sum_{k\in \Lambda}
    \overline{u_{kj}}
    \varphi_k(x),
    \qquad
    j \in \Lambda,\
    x \in G.
\end{equation}
Also, $U:= (u_{jk})_{j,k\in \Lambda} \in \mC^{N\x N}$ is a unitary matrix whose columns are the eigenvectors of $S$, and hence, $S = U\Sigma U^*$, where $\Sigma:= \diag_{j\in \Lambda}(\sigma_j) $ is a diagonal matrix formed from the eigenvalues of $S$. The functions $\theta_j$, defined by (\ref{theta}),  are also orthonormal elements of the Hilbert space $\cH$ in (\ref{cH}), that is, $\bra \theta_j, \theta_k\ket_{\cH} = \delta_{jk}$. The representation (\ref{psi2})  of the PDF $p$ as a convex combination  of the squares of the functions $\sqrt{\nu}|\theta_j|$ involves not only the freedom of varying  the coefficients $\sigma_j$, but also  the possibility  of ``mixing'' the $N$ elements $\varphi_k$  of the given  orthonormal basis of $\cH$ in an arbitrary unitary fashion (\ref{theta}). Since the SDM $S$ is Hermitian and has unit trace,   then the total number of real parameters which enter the PDF $p$ through $S$ is
$
    \dim \mH_N-1 =
    N^2 -1
$.
However, not all of these parameters are active, in general.  From the decomposition
$
    S = \Pi(S)+ \Pi^{\bot}(S)
$
of the matrix $S$ into the orthogonal projections $\Pi(S)$ and $\Pi^{\bot}(S): = S - \Pi(S)$ onto the linear subspace
\begin{equation}
\label{cE}
    \cE
    :=
    \Span\{E_{\ell}:\ \ell \in \mho\},
\end{equation}
spanned by the nonzero structure matrices in (\ref{E}), and its orthogonal complement $\cE^{\bot}$, respectively, it follows that $\Pi^{\bot}(S)$ does not influence the PDF $p$ in (\ref{pES}). Here, $\mho$ is a finite subset of $\Omega$ given by (\ref{mho}). Hence, the number of effective parameters of $p$ does not exceed the dimension $\dim \cE$ of the subspace (\ref{cE}), that is, the rank of an appropriate Gram matrix:
\begin{equation}
\label{dim}
    \dim \cE
    =
    \rank
    (\bra E_{\ell}, E_m\ket)_{\ell, m\in \mho}.
\end{equation}
The following lemma provides a condition for the nonsingularity of the Gram matrix, that is,  the linear independence of the structure matrices $E_\ell$, with $\ell \in \mho$,  in the Hilbert space $\mC^{N\x N}$. To this end, for any finite subset $M \subset \Omega$, we denote by
\begin{equation}
\label{cHM}
    \cH_M
    :=
    \Span
    \{
        \varphi_k:\
        k \in M
    \}
\end{equation}
the subspace of $\cH$ spanned by the corresponding set of basis functions. Also,  let $\cQ_{\Lambda}$ be  a set of functions on $G$ associated with the subspace $\cH_{\Lambda}$ as
\begin{equation}
\label{HH}
    \cQ_{\Lambda}
    :=
    \{f\overline{g}:\ f,g\in \cH_{\Lambda}\}.
\end{equation}
From the algebraic closedness (\ref{tri}) of the basis functions, it follows that
\begin{equation}
\label{HHH}
    \cQ_{\Lambda}\subset \cH_{\mho}.
\end{equation}
Indeed, in view of the notation (\ref{cHM}), any two functions $f,g \in \cH_{\Lambda}$ are representable as $f= a^*\Phi$ and $g= b^*\Phi$ for some  vectors $a,b\in \mC^N$, where $\Phi$ is given by (\ref{Phi}). Therefore, (\ref{PhiPhiPhi}) implies that $f\overline{g} = a^* \Phi\Phi^*b= \sum_{\ell\in \mho} a^*E_{\ell}b \varphi_{\ell} \in \cH_{\mho}$, whence the inclusion (\ref{HHH}) follows in view of the arbitrariness of the functions $f,g \in \cH_{\Lambda}$.

\begin{lem}
\label{lem:HH}
The linear subspace $\cE$, defined by (\ref{cE}), has full dimension $\dim \cE = L$ if and only if the orthogonal complement of the set $\cQ_{\Lambda}$, given by (\ref{HH}), to the subspace $\cH_{\mho}$ consists of the zero function:
\begin{equation}
\label{QH}
    \cQ_{\Lambda}^{\bot}
    \bigcap \cH_{\mho} = \{0\}.
\end{equation}
\end{lem}
\begin{proof}
As mentioned above in regard to (\ref{dim}),
the property $\dim \cE = L$ is equivalent to the linear independence of the structure matrices $E_\ell$, with $\ell \in \mho$, in the space $\mC^{N\x N}$. Now, a  linear combination
$
    \sum_{\ell\in \mho} c_{\ell} E_{\ell}
$ of these matrices
with complex coefficients $c_{\ell}$ vanishes if and only if so does the quantity
$
    a^*
    \sum_{\ell\in \mho}
    c_{\ell} E_{\ell}b
    =
    \sum_{\ell\in \mho}
    c_{\ell}
    \bE_{\nu}\left(\overline{\varphi_{\ell}}\, a^* \Phi\Phi^*b \right)
    =
    \bE_{\nu}\left(f\overline{g}\overline{h}\right)
$
for all $a, b \in \mC^N$. Here, $f:= a^*\Phi \in \cH_{\Lambda}$, $g:= b^*\Phi \in \cH_{\Lambda}$, and $h := \sum_{\ell\in \mho} \overline{c_{\ell}} \varphi_{\ell} \in \cH_{\mho}$ are auxiliary functions. Hence, the matrices $E_{\ell}$, with $\ell \in \mho$, are linearly dependent if and only if there exists a function $h\in \cH_{\mho}\setminus \{0\}$ such that $\bE_{\nu}(f\overline{g}\overline{h})=0$ for all $f,g\in \cH_{\Lambda}$,  that is, $h \in \cQ_{\Lambda}^{\bot}$. This establishes the equivalence between the linear independence of these matrices and the condition (\ref{QH}).
\end{proof}

\section{HERMITE POLYNOMIAL AND FOURIER BASES}\label{sec:Herm}

For completeness, we will now specify the PDFs (\ref{p}) for two classical Hilbert spaces of functions.
Assuming that $G:= \mR^n$, consider the functions
\begin{equation}
\label{h}
    \varphi_k(x)
    :=
    \frac{H_k(x)}{\sqrt{k!}},
    \qquad
    x:= (x_j)_{1\< j\< n} \in \mR^n,
\end{equation}
which are labeled by $n$-indices $k:= (k_j)_{1\< j \< n} \in \mZ_+^n=:\Omega$ (with $\mZ_+$ the set of nonnegative integers) and are obtained from the $n$-variate Hermite polynomials
\begin{equation}
\label{H}
    H_k(x)
    :=
    (-1)^{|k|}
    \re^{\frac{1}{2}|x|^2}
    \d_x^k
    \re^{-\frac{1}{2}|x|^2}.
\end{equation}
Here, use is made of the standard multiindex notation $k!:= k_1!\x \ldots \x k_n!$, $|k|:= k_1 + \ldots + k_n$, and $\d_x^k := \d_{x_1}^{k_1}\ldots \d_{x_n}^{k_n}$, so that $H_k=H_{k_1}\ox \ldots \ox H_{k_n}$ is the tensor product of the corresponding univariate Hermite polynomials. The functions (\ref{h}) form an orthonormal basis in the real Hilbert space $\cH:= \cL^2(\mR^n, \nu)$ of real-valued square integrable functions, where the weight $\nu$ is the $n$-variate standard normal PDF
\begin{equation}
\label{nun}
    \nu(x)
    :=
    \frac{\re^{-\frac{1}{2}|x|^2}}{(2\pi)^{n/2}},
    \qquad
    x \in \mR^n.
\end{equation}
The orthonormal basis $(\varphi_k)_{k\in \mZ_+^n}$, given by  (\ref{h}) and (\ref{H}), satisfies a real-valued version of (\ref{tri}). In view of
\cite[Eq. (3.13) on p. 28]{J_1997} and \cite[Theorem 3.2.1 on p. 13]{M_1997}, the structure coefficients (\ref{e}) are computed as
\begin{equation}
\label{eh}
    e_{jk\ell}
    =
    \frac{\sqrt{j!k!\ell!}}{(m-j)!(m-k)!(m-\ell)!},
    \qquad
    m:=\frac{j+k+\ell}{2},
\end{equation}
for all $n$-indices $j,k,\ell \in \mZ_+^n$ whose entries do not exceed the corresponding entries of $m$ and such that $j+k+\ell$ consists of even integers, with $e_{jk\ell} = 0$ otherwise. In this case, the PDF $p$ in (\ref{p}) is specified by a real symmetric SDM $S$, and its representation  (\ref{pES}) in terms of the structure matrices from (\ref{E}) takes the form
\begin{align}
    p(x)  = \frac{\re^{-\frac{1}{2}|x|^2}}{(2\pi)^{n/2}}\sum_{j,k \in \Lambda}s_{jk}\frac{H_j(x)H_k(x)}{\sqrt{j!k!}}
\label{pherm}
          = \frac{\re^{-\frac{1}{2}|x|^2}}{(2\pi)^{n/2}}\sum_{\ell \in \mho}\Bra S, E_{\ell}\Ket \frac{H_{\ell}(x)}{\sqrt{\ell!}}
\end{align}
in view of (\ref{nun}).
Here, $\Lambda$ is a finite subset of $\mZ_+^n$, and, in view of (\ref{eh}), the corresponding set $\mho$ in (\ref{mho}) is the Minkowski sum of the set $\Lambda$ with itself: $\mho= \Lambda + \Lambda:= \{j+k:\ j,k \in \Lambda\}$. Although (\ref{pherm}) is organised as a truncated Hermite polynomial expansion, its coefficients   $\Bra S, E_{\ell}\Ket$ are parameterized by the SDM $S$ in a specific way, which, according to Lemma~\ref{lem:pS},  guarantees that $p$ is a legitimate PDF. As another example, consider the Fourier basis on the $n$-dimensional torus $\mT^n$ (where $\mT$ is realised as a half-open interval $[0,2\pi)$) consisting of the functions
\begin{equation}
\label{Fourier}
  \varphi_k(x):= \re^{ik^{\rT}x},
  \qquad
  x \in \mT^n,
  \
  k \in \mZ^n.
\end{equation}
These functions are indexed using the $n$-dimensional integer lattice $\mZ^n$ and form an orthonormal basis in the complex Hilbert space $\cH:= \cL^2(\mT^n, (2\pi)^{-n})$ with a constant weight $(2\pi)^{-n}$. The latter is the PDF of the uniform  distribution over the torus. Since the functions (\ref{Fourier})  satisfy $\varphi_j\overline{\varphi_{k}} = \varphi_{j-k}$, the algebraic structure coefficients in (\ref{tri}) are given by
\begin{equation}
\label{efourier}
    e_{jk\ell} = \delta_{j-k,\ell},
    \qquad
    j,k,\ell \in \mZ^n.
\end{equation}
In this case, in view of (\ref{pES}), the PDF $p$ in (\ref{p}) takes the form of a trigonometric polynomial
\begin{align}
    p(x)  =
    \frac{1}{(2\pi)^{n}}
    \sum_{j,k \in \Lambda}
    s_{jk}\re^{i(k-j)^{\rT}x}
\label{pfourier}
        =
        \frac{1}{(2\pi)^{n}}
        \sum_{\ell \in \mho}
        \Bra S, E_{\ell}\Ket
        \re^{i\ell^{\rT}x},
\end{align}
where the set $\mho= \Lambda-\Lambda:=\{j-k:\ j,k\in \Lambda\}$ (which is the Minkowski difference of the set $\Lambda$ with itself) is symmetric about the origin in  $\mZ^n$. Again,
Lemma~\ref{lem:pS} ensures that (\ref{pfourier}) describes a legitimate PDF on the torus for any SDM $S$. In the multivariate case being considered, the trigonometric polynomial $p$ in (\ref{pfourier}) is not necessarily reducible to the squared absolute value of a single trigonometric polynomial.   The Fejer-Riesz theorem  \cite{RS_1955} guarantees such a spectral factorization of a nonnegative trigonometric polynomial only in the univariate case. From the representation (\ref{psi2}), with $\nu=(2\pi)^{-n}$,  it follows that (\ref{pfourier}) is a mixture of $N$ such factorizations:
\begin{equation}
\label{psi2fourier}
    p(x) =
    (2\pi)^{-n}
    \sum_{j\in \Lambda}
    \sigma_j
    |\theta_j(x)|^2,
\end{equation}
where, in accordance with (\ref{theta}),
$
        \theta_j(x)
    :=
    \sum_{k\in \Lambda}
    \overline{u_{kj}}
    \re^{ik^{\rT}x}
$
are trigonometric polynomials. The PDF $p$ in (\ref{psi2fourier}) reduces to a single Fejer-Riesz spectral factorization in the case of a rank-one SDM $S = uu^*$, where $u\in \mC^N$ is a unit complex vector. Recalling the analogy with the quantum mechanical density matrices mentioned above, such SDMs correspond to pure quantum states \cite{S_1994} which are extreme points of the convex set of mixed states. Therefore, this multivariate Fourier basis example shows that the SDM approach is, in principle, able to produce a wider class of PDFs than those  obtained by squaring  a single linear combination of elementary functions.

\section{QUADRATICALLY OPTIMAL SDM APPROXIMATION OF PDFS}\label{sec:SDMapp}

We will now consider a problem of approximating a given PDF $f: G \to \mR_+$ by the  PDF $p$ from  (\ref{p}). More precisely, assuming  that the orthonormal basis in $\cH$ and the set $\Lambda$ are fixed, the SDM $S$ is varied over the set (\ref{fS}) so as to minimize a discrepancy between the actual PDF $f$ and the approximating PDF $p_S:=p$ (which is parameterized by $S$):
\begin{equation}
\label{Dmin}
    D(f,p_S) \longrightarrow \min,
    \qquad
    S \in \fS_N.
\end{equation}
For example, one of such proximity criteria is described by a quadratic functional
\begin{align}
\nonumber
    D(f,p)
    := &
        \frac{1}{2}
        \Big\|\frac{f-p}{\nu}\Big\|_{\cH}^2
        =
    \frac{1}{2}
    \int_G
    \frac{(f(x)-p(x))^2}{\nu(x)} \rd x\\
\label{D}
    =&
    \frac{1}{2}\Big(\re^{\bR(f\|\nu)} + \re^{\bR(p\|\nu)}\Big) - \int_G \frac{f(x)p(x)}{\nu(x)}\rd x,
\end{align}
where use is made of the Renyi entropy from  (\ref{bR}), provided $\bR(f\| \nu)$ and $\bR(p\|\nu)$  are finite, and the $\frac{1}{2}$ factor is introduced for further convenience. A similar, though different, ``mean integrated squared error'' criterion is employed in the kernel density estimation; see, for example,  \cite{BGK_2010,D_2007} and the references therein. In view of the linear dependence of $p_S$ on $S$, the quantity  $D(f,p_S)$, given by (\ref{D}), is a convex (but not necessarily strictly convex) quadratic function of $S$, which makes (\ref{Dmin}) a convex quadratic minimization problem over the convex set of SDMs $\fS_N$. We will therefore consider a ``regularised'' version of this problem:
\begin{equation}
\label{Dminmu}
    D(f,p_S) - \mu \ln\det S \longrightarrow  \min,
    \qquad
    S \in \fS_N,
\end{equation}
where $\mu>0$ is an additional parameter which weights the strictly convex function $-\ln\det S$, with the latter also playing the role of a barrier function for the SDM $S$ to avoid singularity and hence, to remain positive definite. The resulting strictly convex minimization problem (\ref{Dminmu}) has a unique solution which is described below.
The following theorem employs a positive semi-definite self-adjoint operator $\cA$, acting on the Hilbert space $\mH_N$ and completely specified by the structure matrices (\ref{E}) as
\begin{equation}
\label{cA}
    \cA(X):= \sum_{\ell \in \mho} \bra E_{\ell}, X\ket E_{\ell},
    \qquad
    X \in \mH_N.
\end{equation}
We will also need a linear operator $\cB$ which maps a PDF $f: G\to \mR_+$ to a complex Hermitian matrix
\begin{equation}
\label{cB}
    \cB(f)
    : =
    \sum_{\ell\in \mho}
    \bE_f \varphi_{\ell} E_{\ell},
\end{equation}
with the latter being related to the generalized moments of $f$:
\begin{equation}
\label{bEf}
    \bE_f \varphi_{\ell}
    =
    \int_G f \varphi_{\ell} \rd x
    =
    \Big\bra
        \frac{f}{\nu},\,
        \varphi_{\ell}
    \Big\ket_{\cH}.
\end{equation}
The finiteness of these moments is guaranteed by the assumption $\bR(f\| \nu)<+\infty$ in view of the orthonormality (\ref{ortho}) and the Cauchy-Bunyakovsky-Schwarz inequality.

\begin{theorem}
\label{th:Sopt}
Suppose $f: G\to \mR_+$ is a given PDF with finite Renyi entropy in (\ref{bR}): $\bR(f\|\nu )<+\infty$. Then for any given value of the barrier parameter $\mu>0$, the optimization problem (\ref{Dminmu}), with the proximity criterion (\ref{D}), has a unique solution $S\succ 0$ which satisfies
\begin{equation}
\label{Sopt}
    \cA(S) - \mu S^{-1}= \lambda I_N +  \cB(f).
\end{equation}
Here, $\lambda\in \mR$ is found from the normalization condition $\Tr S = 1$, and the operators
$\cA$ and $\cB$ are defined by (\ref{cA})--(\ref{bEf}).
\end{theorem}
\begin{proof}
Application of the method of Lagrange multipliers to the strictly convex minimization problem (\ref{Dminmu}) yields the following condition of optimality:
\begin{align}
\nonumber
    \d_S(D(f,p_S) & - \mu\ln\det S - \lambda (\Tr S-1))\\
\label{dL}
     & = \d_S D(f,p_S)
     - \lambda I_N - \mu S^{-1} = 0.
\end{align}
Here, $\lambda\in \mR$ is  the Lagrange multiplier associated with the trace constraint in (\ref{fS}), and
\begin{equation}
\label{dD}
    \d_S D(f,p_S)
    =
    \int_G
    (p_S-f) \Phi\Phi^* \rd x
\end{equation}
is the Frechet derivative of $D(f,p_S)$ in (\ref{D})  as a composite function of the SDM $S$, where use is made of the pointwise Frechet derivative $
    \d_S p_S = \nu \Phi\Phi^*
$
of  the PDF $p_S$ in (\ref{p}).
By combining (\ref{PhiPhiPhi}) with the moments of the PDF $p_S$ in (\ref{mom}), it follows that
\begin{align}
\label{pmom}
        \int_G p_S \Phi\Phi^* \rd x
        & =
        \sum_{\ell \in \mho}
                \bE_p \varphi_{\ell}
        E_{\ell}
        =
        \cA(S),\\
\label{fmom}
    \int_G f \Phi\Phi^*\rd x
    & =
    \sum_{\ell\in \mho}
    \bE_f \varphi_{\ell} E_{\ell}
    =
    \cB(f).
\end{align}
Here, the operator $\cA$, defined by (\ref{cA}), is the second-order Frechet derivative $\d_S^2D(f,p_S)$,
and use is made of (\ref{cB}), (\ref{bEf}).
Substitution of (\ref{dD})--(\ref{fmom}) into (\ref{dL}) leads to (\ref{Sopt}).  The uniqueness of the pair $(\lambda, S)$, satisfying (\ref{Sopt}) together   with $S\succ 0$ and $\Tr S = 1$, is ensured by strict monotonicity of $\Tr S$ and
$
    \cA(S) - \mu S^{-1}
$
with respect to $S$ in the sense of the partial ordering on $\mH_N$ induced by positive semi-definiteness. Indeed, (\ref{pmom})  implies that the operator $\cA$ is nondecreasing since
$
    \cA(X) = \int_G \nu \bra X, \Phi\Phi^*\ket \Phi\Phi^*\rd x \succcurlyeq 0
$ for any $X \in \mH_N^+$, while the map $S\mapsto -S^{-1}$ is strictly increasing \cite{HJ_2007}.
\end{proof}

The barrier parameter $\mu>0$ in (\ref{Dminmu}) controls sensitivity of the optimal SDM  $S$ to the PDF $f$ being approximated. In particular, $\lim_{\mu\to +\infty} S =\frac{1}{N}I_N$. At the other extreme, for small values of $\mu$, the SDM $S$ can become nearly singular and highly sensitive to $f$.

\section{SDM APPROXIMATION OF PDF DYNAMICS FOR MARKOV DIFFUSION PROCESSES}\label{sec:SDMdyn}

The PDF $p_S$ in  (\ref{p}), with the SDM $S$ computed according to Theorem~\ref{th:Sopt}, can be used as a legitimate approximation for PDFs   of random processes. More precisely, suppose $f(t,\cdot):G \to \mR_+ $ is the time-varying PDF of a $G$-valued Markov diffusion process $\xi(t)$ with an infinitesimal generator $\cG$ (so that $\d_t \bE\varphi(t,\xi(t)) = \bE_f (\d_t\varphi + \cG(\varphi))$ for any smooth test function $\varphi: \mR_+ \x G \to \mC$;  see, for example, \cite{KS_1991}).  For what follows, it is  assumed that the Renyi entropy $\bR(f\| \nu)$  in (\ref{bR}) remains finite. This can be studied by using the integro-differential relations
\begin{equation}
\label{bRdot}
    \d_t \re^{\bR(f\| \nu)}
    =
    \d_t \bE_f \frac{f}{\nu}
    =
    \bE_f
    \Big(
        \frac{\d_t f}{\nu}
        + \cG\Big(\frac{f}{\nu}\Big)
    \Big)
    =
    2\bE_f \cG\Big(\frac{f}{\nu}\Big)
\end{equation}
which follow (provided $f$ and $\nu$ are smooth enough) from the PDF dynamics governed by the FPKE
\begin{equation}
\label{FPKE}
    \d_t f = \cG^{\dagger}(f),
\end{equation}
 where the adjoint $\cG^{\dagger}$ is understood in the sense of the standard $\cL^2$-space with the unit (or constant) weight.
Then  the SDM $S$ of the corresponding approximation $p_S$ of $f$ acquires dependence on time, which can be  represented in the form of an ODE. The theorem below  employs
a self-adjoint operator $F_S$ on the space $\mH_N$, which is the Frechet derivative of the function of $S$ on the left-hand side of (\ref{Sopt}):
\begin{equation}
\label{F}
    F_S(X) = \cA(X) + \mu S^{-1} X S^{-1},
    \qquad
    X \in \mH_N.
\end{equation}
Due to the assumption that $\mu>0$, the operator  $F_S$ is positive definite and strictly increasing for any given   $S\succ 0$. These  properties of $F_S$ are inherited by its inverse $F_S^{-1}$. Also, let $K$ be a linear operator which maps the PDF $f$ to the complex Hermitian matrix
\begin{equation}
\label{K}
      K(f):= \sum_{\ell \in \mho} \bE_f \cG(\varphi_{\ell}) E_{\ell},
\end{equation}
where $\cG$ is the generator of the Markov diffusion process $\xi$, and use is made of the structure matrices (\ref{E}).

\begin{theorem}
\label{th:Sdot}
Suppose the PDF $f(t,\cdot)$ of the underlying Markov diffusion process has finite Renyi entropy $\bR(f\| \nu)$ in (\ref{bR}) at every moment of time $t\> 0$. Then the SDM $S$, associated with $f$ according to Theorem~\ref{th:Sopt}, satisfies the ODE
\begin{equation}
\label{Sdot2}
    \dot{S} = F_S^{-1}(K(f)) - \frac{\Tr F_S^{-1}(K(f))}{\Tr F_S^{-1}(I_N)} F_S^{-1} (I_N),
\end{equation}
where the operators $F_S$ and $K$ are defined by (\ref{F}) and (\ref{K}).
\end{theorem}
\begin{proof} By taking the time derivative on both sides of (\ref{Sopt}), it follows that
\begin{equation}
\label{Sdot}
    F_S(\dot{S}) = \dot{\lambda} I_N + \d_t \cB(f)
    =
    \dot{\lambda} I_N
    +
    K(f),
\end{equation}
where use is made of (\ref{F}) and a combination of (\ref{cB}) with (\ref{K}). The invertibility of the operator $F_S$ allows (\ref{Sdot}) to be solved for $\dot{S}$ as
\begin{equation}
\label{Sdot1}
    \dot{S}
    =
    \dot{\lambda}F_S^{-1} (I_N) + F_S^{-1}(K(f)).
\end{equation}
Since $F_S$ is strictly increasing, then $F_S^{-1} (I_N)\succ 0$. Hence, $\Tr F_S^{-1} (I_N)>0$, and  $\dot{\lambda}$ in (\ref{Sdot1}) can be uniquely found so as to satisfy  the condition $\Tr \dot{S} = 0$ (which comes from the preservation of   $\Tr S = 1$ in time):
\begin{equation}
\label{lamdot}
    \dot{\lambda}
    =
    -\frac{\Tr F_S^{-1}(K(f))}{\Tr F_S^{-1}(I_N)}.
\end{equation}
By substituting (\ref{lamdot}) back into (\ref{Sdot1}), it follows that the SDM $S$ is governed by (\ref{Sdot2}).
\end{proof}

The right-hand side of (\ref{Sdot2}) is linear with respect to the PDF $f$ and depends on $S$ in a rational fashion. The latter follows from the representation
\begin{equation}
\label{Finvvec}
    \vec(F_S^{-1}(X)) = \Big(\sum_{\ell \in \mho} \vec(E_{\ell})\vec(E_{\ell})^* + \mu \overline{S^{-1}}\ox S^{-1}\Big)^{-1} \vec(X)
\end{equation}
obtained by applying the 
vectorization $\vec(\cdot)$ of matrices \cite{M_1988} to (\ref{cA}) and (\ref{F}) and using the relation $S^{\rT}=\overline{S}$ in view of $S$ being Hermitian,
 with $\ox$ the Kronecker product.
If $f$ is an invariant PDF of the Markov process $\xi$, then (\ref{K}) implies that $K(f)=0$, and (\ref{Sdot2}) yields  $\dot{S} = 0$ in accordance with the static setting considered in Section~\ref{sec:SDMapp}. Now, for nonlinear stochastic systems, the solution $f$ of the FPKE (\ref{FPKE}) is usually not available in a closed form. In this case (when approximations of $f$ are particularly important), the right-hand side of (\ref{Sdot2}) can be evaluated by replacing the unknown PDF $f$ with its SDM approximation $p_S$, which leads to
\begin{equation}
\label{SODE}
    \dot{S} = F_S^{-1}(Q(S)) - \frac{\Tr F_S^{-1}(Q(S))}{\Tr F_S^{-1}(I_N)} F_S^{-1} (I_N).
\end{equation}
Here, $Q$ is a linear operator on the space $\mH_N$, obtained by substituting the PDF $p_S$ from (\ref{pES}) into the operator $K$ in (\ref{K}):
\begin{align}
\nonumber
    Q(S)
    & := K(p_S) = \sum_{m \in \mho} \bE_{p_S} \cG(\varphi_m) E_m\\
\label{Q}
    & =
    \sum_{\ell, m \in \mho}
    \Bra
        \varphi_{\ell},
        \cG(\varphi_m)
    \Ket_{\cH}
    \bra
        E_{\ell}, S
    \ket
    E_m.
\end{align}
The right-hand side of the ODE (\ref{SODE}) 
is a rational function of the entries of $S$ (thus resembling the Riccati equations), which can be represented in the vectorised  form by combining (\ref{Finvvec}) with
\begin{equation}
\label{Qvec}
    \vec(Q(S)) =
    \sum_{\ell, m \in \mho}
    \Bra
        \varphi_{\ell},
        \cG(\varphi_m)
    \Ket_{\cH}
    \vec(E_m)
    \vec(E_{\ell})^* \vec(S),
\end{equation}
following from (\ref{Q}). The constant matrices
    in (\ref{Finvvec}) and (\ref{Qvec}) can be precomputed using the structure matrices (\ref{E}) and the matrix elements of the generator $\cG$ over the basis. By construction, the ODE (\ref{SODE})  has a positive definite solution whose trace is preserved in time. Therefore, being a closure of (\ref{Sdot2}),  the SDM dynamics produce a legitimate approximate solution $p_S$ of  the FPKE (\ref{FPKE}).

\section{SMOLUCHOWSKI SDE ON MULTIDIMENSIONAL  TORUS}\label{sec:SSDE}

We will now consider an application of the SDM approach to the approximation of PDFs for random processes governed by nonlinear SDEs. As an illustrative example, we will use a version of the Smoluchowski SDE \cite{GS_2004,KS_1991,Z_2001} on the $n$-dimensional torus $\mT^n$:
\begin{equation}
\label{SSDE}
    \rd \xi(t) = -\nabla V(\xi (t))\rd t + \sigma \rd W(t).
\end{equation}
Here, $\xi := (\xi (t))_{t\> 0}$ is a $\mT^n$-valued Markov diffusion process, and $V: \mT^n \to \mR$ is a twice continuously differentiable   function with the gradient $\nabla V:=(\d_{x_k} V)_{1\< k \< n}: \mT^n \to \mR^n$, where $\d_{x_k}(\cdot)$ is the  partial derivative with respect to the $k$th angular  coordinate. Also,  $\sigma > 0$ is a scalar parameter, and $W$ is an $n$-dimensional standard Wiener process. The state space in the form of a multidimensional torus corresponds to systems with rotational degrees of freedom.   
 The SDE (\ref{SSDE}) (which is understood in the Ito sense) is a noisy version of the gradient descent for the function $V$ and is employed, for example, in the simulated annealing algorithm of stochastic optimization \cite{KGV_1983}.   The function $V$ can be interpreted as the potential energy of a dissipative dynamical system in contact with a heat bath at temperature $T>0$ which specifies the noise level $\sigma$ in (\ref{SSDE}) as
$
    \sigma = \sqrt{2k_{\rB}T}
$,
with $k_{\rB}$ the Boltzmann constant. This interpretation is motivated by the fact \cite{KS_1991} that the invariant probability measure for the Smoluchowski SDE (\ref{SSDE}) is absolutely continuous 
with the Gibbs-Boltzmann PDF
\begin{equation}
\label{inv}
    f_*(x) = \frac{\re^{-\beta V(x)}}{Z(\beta)},
    \qquad
    x \in \mT^n,\
    \beta := \frac{2}{\sigma^2}.
\end{equation}
Here,
$
    Z(\beta):= \int_{\mT^n} \re^{-\beta V(x)}
    \rd x
$ is the statistical mechanical partition function \cite{ME_1981} of the auxiliary parameter $\beta = \frac{1}{k_{\rB} T}$,
 associated with the potential $V$. The invariant PDF $f_*$ in (\ref{inv}) is a unique   equilibrium point of the FPKE (\ref{FPKE}) 
for the PDF $f(t,\cdot): \mT^n \to \mR_+$ of the random vector $\xi (t)$,  which takes the form
\begin{align}
    \d_t f
    & =
    \div( f\nabla V) + \frac{\sigma^2}{2}\Delta f.
\label{SFPKE}
\end{align}
Here,
$\div(\cdot)$ and $\Delta(\cdot)$ are the divergence and Laplace operators over the spatial variables. The corresponding
generator $\cG$ 
acts on a smooth test function $\varphi: \mT^n \to \mC$ as
\begin{equation}
\label{cG}
    \cG(\varphi) = -\nabla V^{\rT} \nabla \varphi + \frac{\sigma^2}{2}\Delta \varphi.
\end{equation}
Although the invariant PDF $f_*$ admits a closed-form representation (\ref{inv}), the time evolution (\ref{SFPKE}) of $f$ towards the equilibrium can be complicated, and it is especially so for multiextremum  potentials $V$. For example,
in the context of molecular dynamics simulation,  such $V$ can represent fractal-like energy landscapes  (as functions of the dihedral angles in protein macromolecules)  which are considered to play an important role in relaxation phenomena  associated  with protein folding \cite{LBO_2010}. This motivates investigation of the PDF dynamics for such FPKEs.

\section{FPKE IN SPATIAL FREQUENCY DOMAIN}\label{sec:spat}

We will now reformulate the FPKE (\ref{SFPKE}) in the spatial frequency domain using
the fact that all functions on the torus $\mT^n$  are $2\pi$-periodic with respect to the angular variables and (under the assumption of square integrability underlying the Hilbert space $\cH:= \cL^2(\mT^n,(2\pi)^{-n})$) can be represented by Fourier series. Let
\begin{equation}
\label{Vf}
    V(x)
    = \sum_{k \in \mZ^n} V_k\varphi_k(x),
    \qquad
    f(t,x)
    = \sum_{k \in \mZ^n} f_k(t)\varphi_k(x)
\end{equation}
be the Fourier series for the potential $V$ and the PDF $f(t,\cdot)$ over the corresponding basis (\ref{Fourier}).
Since the functions $V$ and $f$ are real-valued, their Fourier coefficients satisfy the Hermitian property
$
    V_{-k} = \overline{V_k}$ and
$
    f_{-k} = \overline{f_k}
$ for all $
    k \in \mZ^n
$ and any time $t$. The normalization condition for the PDF $f(t,\cdot)$ is equivalent to
\begin{equation}
\label{f0}
    f_0(t) = (2\pi)^{-n}\int_{\mT^n}f(t,x)\rd x =(2\pi)^{-n},
    \qquad
    t\>0.
\end{equation}
Since the potential $V$ is assumed to be twice continuously differentiable, $\|\Delta V\|_{\infty}
    :=
    \max_{x \in \mT^n} |\Delta V(x)| <+\infty$.
Furthermore, if $V$ is infinitely differentiable, it can be shown that the FPKE (\ref{SFPKE}) has a smooth fundamental solution \cite{S_2008}.
Similarly to the energy estimates for parabolic PDEs \cite{E_1998}, the norm $\|f(t,\cdot)\|_{\cH}$ at any time $t>0$ is amenable to an upper bound in terms of $\|f(0,\cdot)\|_{\cH}$. This is obtained by applying the Gronwall-Bellman lemma to the following differential inequality which is given here for completeness.

\begin{lem}
\label{lem:diss}
The PDF $f$,  governed by the FPKE (\ref{SFPKE}) for the Smoluchowski SDE (\ref{SSDE}) with an infinitely differentiable potential $V$, satisfies
\begin{equation}
\label{diss0}
    \big(\|f\|_{\cH}^2\big)^{^\centerdot}
    \<
    \big(
        \|\Delta V\|_{\infty} - \sigma^2
    \big)
    \|f\|_{\cH}^2
    +
    (2\pi)^{-2n}
    \sigma^2.
\end{equation}

\end{lem}
\begin{proof}
A combination of (\ref{SFPKE}), (\ref{cG}) with the identities
$
    f\nabla V^{\rT} \nabla f
    =
    \frac{1}{2}
    \big(
        \div(f^2 \nabla V) - f^2 \Delta V
    \big)$ and
$
    f\Delta f = \frac{1}{2} \Delta(f^2) - |\nabla f|^2
$
(and the property that $f$ and $V$ are real-valued) leads to 
\begin{align}
\nonumber
    \big(\|f\|_{\cH}^2\big)^{^\centerdot}
    & =
    2
    \Bra
        f,
        \cG^{\dagger}
        (f)
    \Ket_{\cH}
    =
    2
    \Bra
        \cG(f), f
    \Ket_{\cH}\\
\nonumber
    & =
    \bra
        \Delta V, f^2
    \ket_{\cH}
    -
    \sigma^2
    \|\nabla f\|_{\cH}^2\\
\label{diss}
    & \<
    \|\Delta V\|_{\infty}
    \|f\|_{\cH}^2 - \sigma^2\|g\|_{\cH}^2.
\end{align}
Here, the first two equalities are, in fact, a particular case of (\ref{bRdot}) since the Hilbert space $\cH$ being considered employs a constant weight. Also, use is made of the Poincare-Wirtinger inequality
\begin{align}
\nonumber
    \|\nabla f\|_{\cH}^2
    & = (2\pi)^{-n}\int_{\mT^n} |\nabla f|^2\rd x
       = \sum_{k\in \mZ^n\setminus \{0\}} |f_k|^2 |k|^2\\
\label{fff}
    & \>
    \|g\|_{\cH}^2 = \|f\|_{\cH}^2 - f_0^2,
\end{align}
applied to the function $g:= f - f_0$ which satisfies $\int_{\mT^n} g \rd x = 0$ in view of (\ref{f0}). Substitution of the last equality from (\ref{fff}) into (\ref{diss}) leads to (\ref{diss0}).
\end{proof}

Another upper bound for the Renyi entropy $\bR(f\| \nu)$ is provided by the inequalities
\begin{align*}
    \frac{\re^{\bR(f\| \nu)}}{(2\pi)^n\max_{x \in \mT^n} f_*(x)}
    & \< \re^{\bR(f\| f_*)}
    \<
    \re^{\bR(f(0,\cdot)\| f_*)}\\
      & \<     \frac{\re^{\bR(f(0,\cdot)\| \nu)}}{(2\pi)^n\min_{x \in \mT^n} f_*(x)},
\end{align*}
where both denominators are finite and strictly positive due to (\ref{inv}) and the continuity  of $V$ over the torus.
The second of these inequalities follows from the dissipation relation
\begin{align}
\nonumber
\d_t \re^{\bR(f\| f_* )}
& = 2\bE_f \cG(h)
  = 2\bE_{f_*} (h\cG(h))\\
\nonumber
  & =
  \int_{\mT^n}
  \cG^{\dagger}(f_*) h^2\rd x - \sigma^2 \bE_{f_*}(|\nabla h|^2)\\
\label{bRff}
  & =- \sigma^2 \bE_{f_*}(|\nabla h|^2) \< 0,
\end{align}
whereby $\bR(f(t,\cdot)\| f_*)$ is a nonincreasing function of time $t\>0$. Here,
$h:=\frac{f}{f*} $ is the PDF of the Markov process $\xi$ with respect to the invariant measure (so that the first equality in (\ref{bRff}) is similar to (\ref{bRdot})), and use is made of a change of measure together with the relations  $\cG^{\dagger}(f_*) = 0$ and $2h\cG(h) = \cG(h^2) - \sigma^2|\nabla h|^2$, with the latter following from (\ref{cG}).
Now, (\ref{Vf}) allows the matrix elements of the generator $\cG$ over the Fourier basis (\ref{Fourier}) to be computed as
\begin{equation}
\label{cGmat}
    \bra
        \varphi_{\ell},
        \cG(\varphi_m)
    \ket_{\cH}
    =
    V_{\ell-m}
    (\ell-m)^{\rT} m - \frac{\sigma^2}{2} \delta_{\ell m} |m|^2
\end{equation}
for all $\ell, m \in \mZ^n
$.
Hence, in addition to (\ref{f0}), the Fourier coefficients of the PDF $f$ satisfy a denumerable set of ODEs
\begin{align}
\nonumber
    \mathop{f_j}^{\centerdot}
    & =
    \Bra
        \varphi_j, \cG^{\dagger}(f)
    \Ket_{\cH}
    =
    \sum_{k \in \mZ^n}
    \Bra
        \cG(\varphi_j), \varphi_k
    \Ket_{\cH}
    f_k\\
\label{freqODE}
    & =
    -j^{\rT}\sum_{k\in \mZ^n} V_k k f_{j-k}
      - \frac{\sigma^2}{2}|j|^2 f_j,
    \qquad
    j \in \mZ^n\setminus \{0\},
\end{align}
which represent the FPKE (\ref{SFPKE}) in the spatial frequency domain. Here, the convolution sum 
comes from the term $f\nabla V$ 
and is responsible for the coupling of these ODEs (the trivial case of a constant potential $V$ is not considered).
The truncation of the set of ODEs (\ref{freqODE}) (for example, in accordance with Galerkin's method for parabolic PDEs \cite{E_1998}) leads to an approximate system of equations whose solution is not necessarily nonnegative and does not correspond to a legitimate PDF on the torus. This issue can be overcome by using
the SDM approximation of PDF dynamics described in Section~\ref{sec:SDMdyn}.

\section{SDM APPROXIMATION FOR SMOLUCHOWSKI SDE}\label{sec:smolSDM}

In application to the Smoluchowski SDE (\ref{SSDE}) being considered, the SDM dynamics (\ref{SODE}) can be implemented in the vectorised form
\begin{equation}
\label{SODEvec}
    \vec(S)^{^\centerdot} = \vec(F_S^{-1}(Q(S))) - \frac{\Tr F_S^{-1}(Q(S))}{\Tr F_S^{-1}(I_N)} \vec(F_S^{-1} (I_N))
\end{equation}
using (\ref{Finvvec}), (\ref{Qvec}) together with the structure coefficients (\ref{efourier}) and the matrix elements (\ref{cGmat}) over the Fourier basis  (\ref{Fourier}). The corresponding SDM approximation (\ref{pfourier}) of the PDF $f$ in (\ref{SFPKE}) is
\begin{equation}
\label{pfour}
    p_S(x)
        =
        (2\pi)^{-n}C(S)^*\Psi(x),
\end{equation}
where 
the maps $C$ in (\ref{C}) and $\Psi: \mT^n \to \mC^L$  in (\ref{Psi}) are given by 
\begin{equation}
\label{CPsi}
    C(S)
    =
    \Big(
        \sum_{j,k\in \Lambda:\, j-k = \ell} s_{jk}
    \Big)_{\ell \in \mho},
    \qquad
    \Psi(x) = \big(\re^{i\ell^{\rT}x}\big)_{\ell\in \mho }.
\end{equation}
Here, $N= (2r+1)^n$ and $L= (4r+1)^n$ in the case when the set $\Lambda$ and the associated set $\mho$ in (\ref{mho}) are the discrete cubes
\begin{equation}
\label{sets}
    \Lambda:= ([-r,r]\bigcap \mZ)^n,
    \qquad
    \mho= ([-2r,2r]\bigcap \mZ)^n,
\end{equation}
with $r$ a positive integer. We will now provide results of a numerical experiment on the SDM approximation of PDFs for the Smoluchowski SDE (\ref{SSDE}) with $\sigma= 1$ in the two-dimensional case $n=2$. The potential $V$ was generated as a trigonometric polynomial
$
    V(x) := \sum_{k\in \mZ^2:\, |k|\< R} V_k \re^{ik^{\rT}x}
    =
    \sum_{k\in \mZ^2:\, |k|\<R} |V_k| \cos(k^{\rT} x + \phi_k)
$
of $x\in \mT^2$
with $R= 5$ (that is, $41$ independent harmonics) and  exponentially distributed random amplitudes $|V_k|= |V_{-k}|$ and initial phases $\phi_k:= \arg V_k = -\phi_{-k}$, with the latter being uniformly distributed over the interval $[0,2\pi)$ for any $k\ne 0$; see Fig.~\ref{fig:V}.
\begin{figure}[thpb]
      \centering
      \includegraphics[width=100mm, height=60mm]{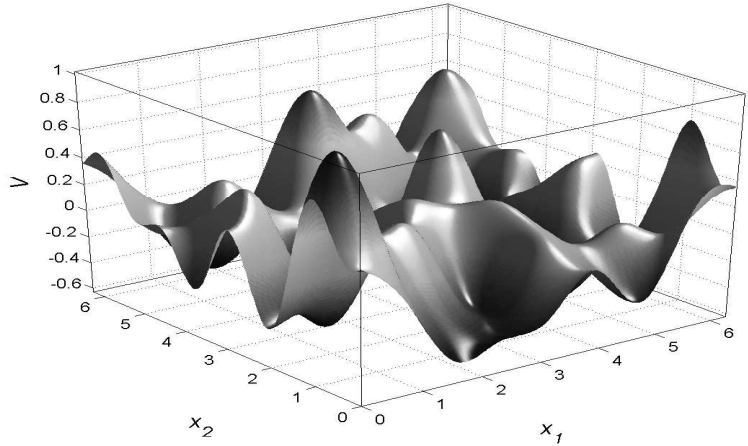}
      \caption{A multiextremum potential $V$ generated as a trigonometric polynomial of the spatial variables $0\< x_1, x_2<2\pi$ with random coefficients.}
      \label{fig:V}
   \end{figure}
 The SDM approximation (\ref{pfour}) was produced using (\ref{SODEvec}) and (\ref{CPsi}) with the sets (\ref{sets}) with $r=2$, $N=(2r+1)^2=25$, $L = (4r+1)^2 = 81$ and $\mu = 0.01$. The initial conditions for the actual PDF $f$ and the SDM $S$ were  $f(0,\cdot)=\frac{1}{4\pi^2}$ and $S(0) = \frac{1}{N}I_N$, with the latter corresponding to the same uniform distribution over the torus with the PDF  $p=\frac{1}{4\pi^2}$.
 The actual PDF $f$ was obtained through the standard finite-difference numerical solution of the FPKE on a mesh of $10^4$ points with time step $0.002$. The evolution of the relative values $\frac{D(f,p_S)}{D(f,0)}$ of the quadratic proximity criterion (\ref{D}) (with $D(f,0)$ being the squared  $\cL^2$-norm of $f$) is presented in Fig.~\ref{fig:DD} which shows that the relative error remained within $1.4\%$.
  \begin{figure}[thpb]
      \centering
      \includegraphics[width=100mm, height=60mm]{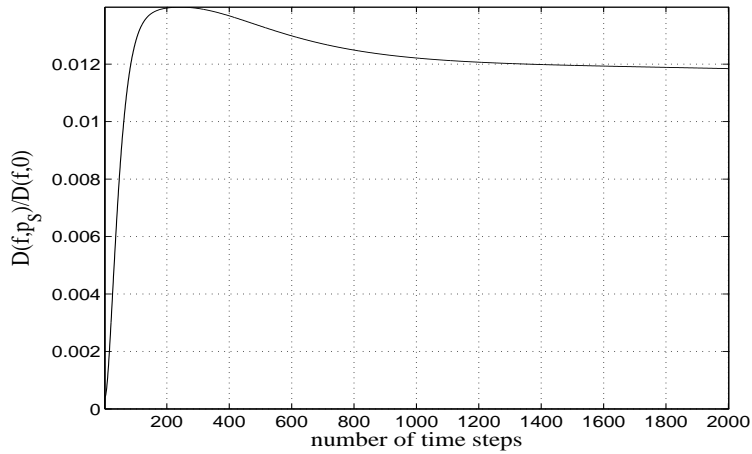}
      \caption{The relative error $\frac{D(f,p_S)}{D(f,0)}$ of the SDM approximation in the course of evolution of the actual PDF $f$.}
      \label{fig:DD}
   \end{figure}
   That the time interval captures the essential part of the movement of the system towards equilibrium is seen from Fig.~\ref{fig:fff}, whereby both the actual PDF $f$ and its SDM approximation $p_S$ at the end of simulation were close to the invariant PDF $f_*$.
  \begin{figure}[thpb]
      \centering
      \includegraphics[width=100mm, height=60mm]{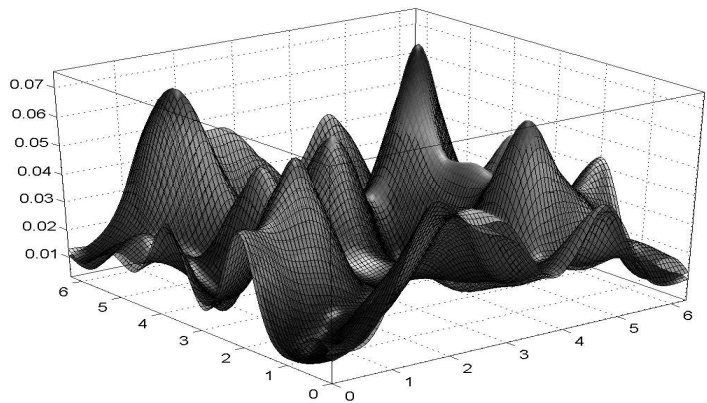}
      \caption{The invariant PDF $f_*$ (semitranslucent surface), the actual PDF $f$ (black wireframe) and its SDM approximation $p_S$ (red wireframe) after 2000 steps of evolution.}
      \label{fig:fff}
   \end{figure}

\section{CONCLUSION}\label{sec:conc}

We have outlined an approach to the approximation of PDFs by quadratic forms of weighted orthonormal basis functions parameterised by SDMs.  The SDM approximation produces a legitimate PDF which satisfies the normalization condition and is nonnegative everywhere. For orthonormal bases with an algebraic structure,  we have provided  an optimization procedure  for the SDM approximation using a quadratic criterion. The optimal SDM approximation has been extended to PDFs of Markov diffusion processes, and an  
ODE has been obtained for the SDM dynamics which yields a legitimate approximate solution of the FPKE. This has been demonstrated for the Smoluchowski SDE with a multiextremum potential on a torus. The SDM approach is also applicable (in the spirit of projective filtering \cite{VM_2005})   to the approximation of posterior PDFs governed by the Kushner-Stratonovich equations \cite{K_1964}. In this application, the conditional SDM would play a part similar to that of the covariance matrix computed in the Kalman filters. However, unlike the covariance matrices, the SDM is not restricted to the second moments and involves higher order moments of the system variables. Numerical integration of the SDM dynamics can be implemented in a square-root form using the smoothness of the Cholesky factorization of positive definite matrices \cite{S_1995}. Error analysis of the SDM approximation is another line of research to be tackled in future publications.

\end{document}